\def\bu{\bullet}
\def\marker{\>\hbox{${\vcenter{\vbox{
    \hrule height 0.4pt\hbox{\vrule width 0.4pt height 6pt
    \kern6pt\vrule width 0.4pt}\hrule height 0.4pt}}}$}\>}
\def\gpic#1{#1
     \smallskip\par\noindent{\centerline{\box\graph}} \medskip}

\def\bu{\bullet}
\def\marker{\>\hbox{${\vcenter{\vbox{
    \hrule height 0.4pt\hbox{\vrule width 0.4pt height 6pt
    \kern6pt\vrule width 0.4pt}\hrule height 0.4pt}}}$}\>}
\def\gpic#1{#1
     \smallskip\par\noindent{\centerline{\box\graph}} \medskip}
\documentclass[12pt]{article}
\usepackage{array,amsmath,amssymb,amsthm}  				    
\usepackage{fullpage,lineno}
\renewcommand{\baselinestretch}{1.15}		

\begin{document}


\input epsf.tex
\def\sp{\bigskip}
\def\ti{\\ \hglue \the \parindent}
\def\ce#1{\LP\medskip\centerline{#1}\medskip}
\def\LP{\par\noindent}

\newtheorem{theorem}{Theorem}[section]
\newtheorem*{theorem*}{Theorem}
\newtheorem{lemma}[theorem]{Lemma}
\newtheorem{corollary}[theorem]{Corollary}
\newtheorem{prop}[theorem]{Proposition}
\newtheorem{conjecture}[theorem]{Conjecture}
\newtheorem{observation}[theorem]{Observation}
\newtheorem{claim}[theorem]{Claim}
\newtheorem{alg}[theorem]{Algorithm}
\theoremstyle{definition}
\newtheorem{remark}[theorem]{Remark}
\newtheorem{example}[theorem]{Example}
\newtheorem{definition}[theorem]{Definition}
\def\PF{\LP{\it Proof.}}
\def\qed{\ifhmode\unskip\nobreak\hfill$\Box$\bigskip\fi \ifmmode\eqno{Box}\fi}

\def\al{\alpha} \def\be{\beta}  \def\ga{\gamma} \def\dlt{\delta}
\def\eps{\epsilon} \def\th{\theta}  \def\ka{\kappa} \def\lmb{\lambda}
\def\sg{\sigma} \def\om{\omega}
\def\nul{\varnothing} 
\def\st{\colon\,}   
\def\MAP#1#2#3{#1\colon\,#2\to#3}
\def\VEC#1#2#3{#1_{#2},\ldots,#1_{#3}}
\def\VECOP#1#2#3#4{#1_{#2}#4\cdots #4 #1_{#3}}
\def\SE#1#2#3{\sum_{#1=#2}^{#3}}  \def\SGE#1#2{\sum_{#1\ge#2}}
\def\PE#1#2#3{\prod_{#1=#2}^{#3}} \def\PGE#1#2{\prod_{#1\ge#2}}
\def\UE#1#2#3{\bigcup_{#1=#2}^{#3}}
\def\CH#1#2{\binom{#1}{#2}} \def\MULT#1#2#3{\binom{#1}{#2,\ldots,#3}}
\def\FR#1#2{\frac{#1}{#2}}
\def\FL#1{\left\lfloor{#1}\right\rfloor} \def\FFR#1#2{\FL{\frac{#1}{#2}}}
\def\CL#1{\left\lceil{#1}\right\rceil}   \def\CFR#1#2{\CL{\frac{#1}{#2}}}
\def\Gb{\overline{G}}
\def\NN{{\mathbb N}} \def\ZZ{{\mathbb Z}} 
\def\esub{\subseteq}
\def\cD{{\mathcal D}}

\def\B#1{{\bf #1}}      \def\R#1{{\rm #1}}
\def\I#1{{\it #1}}      \def\c#1{{\cal #1}}
\def\C#1{\left | #1 \right |}    
\def\CC#1{\left \Vert #1 \right \Vert}    
\def\P#1{\left ( #1 \right )}    

\def\er{{\rm e}}

\title{Degree lists and connectedness are $3$-reconstructible for graphs
with at least seven vertices}

\author{ 
Alexandr V. Kostochka\thanks{University of Illinois at Urbana--Champaign,
Urbana IL 61801, and Sobolev Institute of Mathematics, Novosibirsk 630090,
Russia: \texttt{kostochk@math.uiuc.edu}.  Research supported in part by NSF
grants DMS-1600592 and grants 18-01-00353A and 16-01-00499  of the Russian
Foundation for Basic Research.}\,,
Mina Nahvi\thanks{University of Illinois at Urbana--Champaign,
Urbana IL 61801: \texttt{mnahvi2@illinois.edu}.}\,,
Douglas B. West\thanks{Zhejiang Normal University, Jinhua, China 321004
and University of Illinois at Urbana--Champaign, Urbana IL 61801:
\texttt{dwest@math.uiuc.edu}.
Research supported by National Natural Science Foundation of China grant NNSFC
11871439.}\,,
Dara Zirlin\thanks{University of Illinois at Urbana--Champaign, Urbana IL 61801:
\texttt{zirlin2@illinois.edu}.}
}
\date{\today}
\maketitle

\begin{abstract}
The {\it $k$-deck} of a graph is the multiset of its subgraphs induced by $k$
vertices.  A graph or graph property is {\it $l$-reconstructible} if it is
determined by the deck of subgraphs obtained by deleting $l$ vertices.  We show
that the degree list of an
$n$-vertex graph is $3$-reconstructible when $n\ge7$, and the threshold on $n$
is sharp.  Using this result, we show that when $n\ge7$ the $(n-3)$-deck also
determines whether an $n$-vertex graph is connected; this is also sharp.
These results extend the results of Chernyak and Manvel, respectively, that the
degree list and connectedness are $2$-reconstructible when $n\ge6$, which are
also sharp.

\end{abstract}

\baselineskip16pt

\section{Introduction} \label{sec:intro}
A {\it card} of a graph $G$ is a subgraph of $G$ obtained by deleting one
vertex.  Cards are unlabeled, so only the isomorphism class of a card is given.
The {\it deck} of $G$ is the multiset of all cards of $G$.  A graph is
{\it reconstructible} if it is uniquely determined by its deck. 
The famous Reconstruction Conjecture was first posed in 1942.

\begin{conjecture}[The Reconstruction Conjecture; Kelly~\cite{kel1,kel2},
Ulam~\cite{U}]
Every graph having more than two vertices is reconstructible. 
\end{conjecture}

The two graphs with two vertices have the same deck.  Graphs in many families
are known to be reconstructible; these include disconnected graphs, trees,
regular graphs, and perfect graphs.  Surveys on graph reconstruction
include~\cite{Bondy91,BondyHemminger77,Lauri97,LS,Maccari02}.

Various parameters have been introduced to measure the difficulty of 
reconstructing a graph.  Harary and Plantholt~\cite{HP} defined the
{\it reconstruction number} of a graph to be the minimum number of cards from
its deck that suffice to determine it, meaning that no other graph has the same
multiset of cards in its deck (surveyed in~\cite{AFLM,Mthe}).  Kelly looked in
another direction, considering cards obtained by deleting more vertices.
He conjectured a more detailed version of the Reconstruction Conjecture.

\begin{conjecture}[Kelly~\cite{kel2}]
For $l\in\NN$, there is an integer $M_l$ such that any graph with at
least $M_l$ vertices is reconstructible from its deck of cards obtained by
deleting $l$ vertices.
\end{conjecture}

\noindent
The original Reconstruction Conjecture is the claim $M_1=3$.

A {\it $k$-card} of a graph is an induced subgraph having $k$ vertices.
The \emph{$k$-deck} of $G$, denoted $\cD_k(G)$, is the multiset
of all $k$-cards.  When discussing reconstruction from the $k$-deck, we will
refer to $k$-cards simply as cards.

\begin{definition}
A graph $G$ is {\it $k$-deck reconstructible} if $\cD_k(H)=\cD_k(G)$ implies
$H\cong G$.  A graph $G$ (or a graph invariant) is {\it $l$-reconstructible}
if it is determined by $\cD_{|V(G)|-l}(G)$ (agreeing on all graphs having
that deck).  The {\it reconstructibility} of $G$, written $\rho(G)$, is the
maximum $l$ such that $G$ is $l$-reconstructible.
\end{definition}

For an $n$-vertex graph, ``$k$-deck reconstructible'' and
``$l$-reconstructible'' have the same meaning when $k+l=n$.
Kelly's conjecture is that for any $l\in\NN$, all sufficiently large graphs
are $l$-reconstructible.  Let $K'_{1,3}$ and $K''_{1,3}$ be the graphs
obtained from the claw $K_{1,3}$ by subdividing one or two edges, respectively.
The $5$-vertex graphs $C_4+K_1$ and $K'_{1,3}$ are not $2$-reconstructible,
since they have the same $3$-deck.  Having checked by computer that every graph
with at least six and at most nine vertices is $2$-reconstructible, McMullen
and Radziszowski~\cite{MR} asked whether $M_2=6$.  With computations up to nine
vertices, Rivshin and Radziszowski~\cite{RR} conjectured $M_l\le3l$.

Some results about reconstruction have been extended to the context of
reconstruction from the $k$-deck.  For example, almost every graph is
reconstructible from any set of three cards in the deck of cards obtained by
deleting one vertex (see~\cite{Bol,Chi,Mul}).  Spinoza and West~\cite{SW}
proved more generally that for $l=(1-o(1))\C{V(G)}/2$, almost all graphs are
$l$-reconstructible using only $\CH{l+2}2$ cards that omit $l$ vertices.  Among
other results, they also determined $\rho(G)$ exactly for every graph $G$ with
maximum degree at most $2$.

Since each induced subgraph with $k-1$ vertices arises exactly $n-k+1$ times by
deleting one vertex from a member of $\cD_k(G)$, we have the following.

\begin{observation}\label{k-1}
For any graph $G$, the $k$-deck $\cD_k(G)$ determines the $(k-1)$-deck
$\cD_{k-1}(G)$.
\end{observation}

By Observation~\ref{k-1}, information that is $k$-deck reconstructible is also
$j$-deck reconstructible when $j>k$.  This motivates the definition of
reconstructibility; if $G$ is $l$-recon\-struct\-ible, then $G$ is also
$(l-1)$-reconstructible, so we seek the largest such $l$.

Manvel~\cite{Manvel} proved for $n\ge6$ that the $(n-2)$-deck of an $n$-vertex
graph determines whether the graph satisfies the following properties:
connected, acyclic, unicyclic, regular, and bipartite.  For the first three of
these properties, sharpness of the threshold on $n$ is shown by the graphs
$C_4+P_1$ and $K'_{1,3}$ mentioned above.  Spinoza and West~\cite{SW} extended
Manvel's result by showing that connectedness is $3$-reconstructible when
$n\ge25$.  Using a somewhat different approach, we extend their result.

\begin{theorem}\label{conn}
For $n\ge7$, connectedness is $3$-reconstructible for $n$-vertex graphs,
and the threshold on $n$ is sharp.
\end{theorem}

The threshold is sharp because $C_5+P_1$ and $K''_{1,3}$ have the same $3$-deck.
For general $l$, the known upper and lower bounds on the threshold for $n$ to
guarantee that connectedness of $n$-vertex graphs is $l$-reconstructible are
quite far apart.  Spinoza and West~\cite{SW} proved that connectedness is
$l$-reconstructible when $n>2l^{(l+1)^2}$.  As a lower bound, we know only that
$n>2l$ is needed, since $C_{l+1}+P_{l-1}$ and $P_{2l}$ have the same
$l$-deck~\cite{SW}.  Indeed, $P_n$ is the only $n$-vertex graph whose
reconstructibility is known to be less than $n/2$.

One of the first easy results in ordinary reconstruction is that the degree
list of a graph with at least three vertices is $1$-reconstructible.
Manvel~\cite{Manvel} showed that the degree list is reconstructible from the
$k$-deck when the maximum degree is at most $k-2$.  With no restriction on the
maximum degree, Taylor showed that the degree list is reconstructible from the
$k$-deck when the number of vertices is not too much larger than $k$,
regardless of the value of the maximum degree.

\begin{theorem}[Taylor~\cite{Taylor}]
If $l\ge3$ and $n\ge g(l)$, then the degree list of any $n$-vertex graph is
determined by its $(n-l)$-deck, where
$$
g(l) = (l+\log{l}+1)
\left( \er + \frac{\er \log{l} + \er +1}{(l-1) \log{l}-1} \right) +1
$$
and $\er$ denotes the base of the natural logarithm.  Thus the degree list is
$l$-reconstructible when $n>\er l+O(\log l)$.
\end{theorem}

\noindent
For small $l$, one can obtain exact thresholds.  Chernyak~\cite{Che} proved
that the degree list is $2$-reconstructible when $n\ge6$; again the example
of $C_5+P_1$ and $K'_{1,3}$ shows that this is sharp.  We extend this
to $3$-reconstructibility.

\begin{theorem}\label{degn3}
For $n\geq7$, any two graphs of order $n$ that have the same $(n-3)$-deck have
the same degree list, and this threshold on $n$ is sharp.
\end{theorem}

Again the example of $C_5+P_1$ and $K''_{1,3}$ proves sharpness.
We use Theorem~\ref{degn3} as a tool in the proof of Theorem~\ref{conn}.  With
Chernyak's result being somewhat inaccessible, we also obtain it and Manvel's
result on $2$-reconstructibility of connectedness as corollaries of our
results.

\section{$3$-reconstructibility of degree lists}

We begin with a basic counting tool used also by Manvel~\cite{Manvel} and by
Taylor~\cite{Taylor}.  In a graph $G$, we refer to a vertex of degree $j$ as a
{\it $j$-vertex}.  

\begin{lemma}\label{deg}
Let $\phi(j)$ denote the total number of $j$-vertices over
all cards in the $k$-deck $\cD_k$ of an $n$-vertex graph $G$.  Letting $a_i$
denote the number of $i$-vertices in $G$ (and $l=n-k$),
\begin{equation}\label{l1}
\phi(j)=\SE ij{j+l} a_i\CH ij\CH{n-1-i}{k-1-j} .
\end{equation}
\end{lemma}

\begin{proof}
In each card, each vertex counted by $\phi(j)$ has degree at least $j$ in $G$.
When that degree is $i$, the vertex in the reconstructed graph contributes
exactly $\CH ij\CH{n-1-i}{k-1-j}$ to the computation of $\phi(j)$.  This
contribution is $0$ when $k-1-j> n-1-i$; the vertex then does not have enough
nonneighbors in the full graph to occur with degree exactly $j$ in a card.
Thus we require $i\le n-k+j=l+j$.
\end{proof}

\begin{corollary}[Manvel~\cite{Manvel}]\label{manvel}
From the $k$-deck of a graph and the numbers of vertices with degree $i$
for all $i$ at least $k$, the degree list of the graph is determined.
\end{corollary}
\begin{proof}
Since the $k$-deck determines the $(k-1)$-deck, using induction it suffices to
show that knowing both $\cD_k(G)$ and $a_i$ for $i\ge k$ determines $a_{k-1}$.
Simply solve for $a_{k-1}$ in the expression~\eqref{l1} for $\phi(k-1)$
obtained by setting $j=k-1$.
\end{proof}

With these tools, we prove Theorem \ref{degn3}, which we restate.
\begin{theorem*}[{\bf 1.7}]
For $n\geq7$, any two graphs of order $n$ that have the same $(n-3)$-deck have
the same degree list, and this threshold on $n$ is sharp.
\end{theorem*}
\begin{proof}
For sharpness, the $3$-decks of both $C_5+K_1$ and $K''_{1,3}$ consist of five
copies of $P_3$, ten copies of $P_2+P_1$, and five copies of $3P_1$.

Given $n\ge7$, let $\cD$ be the $(n-3)$-deck of an $n$-vertex graph.
We show that all reconstructions from $\cD$ have the same degree list.

Let $G$ and $H$ be reconstructions from $\cD$.  Since $\cD$ determines
the $2$-deck, we know the common number of edges in $G$ and $H$; let it be
$m$.  We may assume $m\le\FR12\CH n2$, since otherwise we can analyze
the complements of $G$ and $H$. 

We will use repeatedly the fact that any $t$ vertices whose degrees sum to at
least $s$ are together incident to at least $s-\CH t2$ edges.

\nobreak
Let $a_i$ and $b_i$ be the numbers of $i$-vertices in $G$ and
$H$, respectively, and let $c_i=a_i-b_i$.  The computation in~\eqref{l1} is
valid using either $G$ or $H$, producing the same value $\phi(j)$ from $\cD$.
Hence the difference of the two instances of~\eqref{l1} yields
\begin{equation}\label{phidiff}
0=\SE ij{j+3} c_i\CH ij\CH{n-1-i}{n-4-j},
\end{equation}
since here $k=n-3$.  We will be interested in particular in the cases
$j=n-4$ (dominating vertices on cards) and $j=n-5$, which we write explicitly
as

\begin{equation}\label{j4}
c_{n-4}+(n-3)c_{n-3}+\CH{n-2}{2}c_{n-2}+\CH{n-1}{3}c_{n-1}=0
\end{equation}
and
\begin{equation}\label{j5}
4c_{n-5}+3(n-4)c_{n-4}+2\binom{n-3}{2}c_{n-3}+\binom{n-2}{3}c_{n-2}=0.
\end{equation}

The observation of Manvel (Corollary~\ref{manvel}) implies that if $G$ and $H$
have different degree lists, then $c_i\ne 0$ for some $i$ with $i\ge n-3$.  Let
$h$ be the largest such index.  By symmetry, we may assume $c_h<0$.  We consider
cases depending on the value of $h$. 

\bigskip
{\bf Case 1:} $h=n-3$.
In this case $c_{n-1}=c_{n-2}=0$ and $c_{n-3}<0$.  By~\eqref{j4},
$c_{n-4}+(n-3)c_{n-3}=0$.  Since $2(n-3)>n$ when $n\ge7$, we have $c_{n-3}=-1$
and $c_{n-4}=n-3$.  Now~\eqref{j5} implies $c_{n-5}=-(n-3)(n-4)/2$.
Thus $H$ has at least $1+(n-3)(n-4)/2$ vertices, but
$n\ge 1+(n-3)(n-4)/2$ requires $n\le 7$.  Hence $n=7$ and $H$ has degree list
exactly $(4,2,2,2,2,2,2)$, and $G$ has no vertices of degree $2$ or at least
$4$.  Furthermore $c_{n-4}=n-3$, so $G$ has exactly four vertices with 
degree $3$ and cannot reach the same degree-sum as $H$.

\bigskip
{\bf Case 2:} $h=n-2$.
Now $c_{n-1}=0$ and $c_{n-2}<0$.  Let $c_{n-2}=-r$.
By~\eqref{j4}, $c_{n-4}+{(n-3)c_{n-3}}=r\CH{n-2}{2}$,
so $\FR{c_{n-4}}{n-3}+c_{n-3}=(n-2)\FR r2$.  With $r\ge2$ and $n\ge7$ and
$c_{n-4}+c_{n-3}\le n$, this can only be satisfied when $r=2$, $c_{n-3}=n-2$,
and $c_{n-4}=0$.  Since $m\le \FR12\CH n2$, the degree-sum is at most $\CH n2$;
hence $(n-2)(n-3)\le \FR12 n(n-1)$, which requires $n<8$.  Since we have
obtained $(c_5,c_4,c_3)=(-2,5,0)$, \eqref{j5} yields
$4c_2=\CH 53\cdot 2-2\CH 42\cdot 5=-40$; this
requires $c_2=-10$, a contradiction when $n=7$.  Hence we conclude $r=1$.

With $r=1$, we have $c_{n-4}+(n-3)c_{n-3}=\binom{n-2}{2}$.
Hence
\begin{equation}\label{cn4}
c_{n-3}=\FR{n-2}2-\FR{c_{n-4}}{n-3}.
\end{equation}
Substituting into~\eqref{j5} yields
\begin{align}\label{cn5}
4c_{n-5}&=\CH{n-2}3-2\CH{n-3}2c_{n-3}-3(n-4)c_{n-4}\nonumber\\
&=-\FR{(n-2)(n-3)(n-4)}3-2(n-4)c_{n-4}
\end{align}
Since $c_{n-3}$ must be an integer, by~\eqref{cn4} there are not many
possibilities for $c_{n-4}$.  Let $t=\FR{c_{n-4}}{n-3}$.  Since
$\C{c_{n-4}}\le n$, we have $t\in \{1,0,-1\}$ when $n$ is even, and
$t\in\{1/2,-1/2\}$ when $n$ is odd.  Also~\eqref{cn5} simplifies to
$-c_{n-5}=\FR{(n-3)(n-4)}{12}[n-2+6t]$.

With $c_{n-5}\ge -n$ and $n\ge7$, the possibilities that remain for $(n,t)$
are $(7,-1/2)$, $(8,-1)$, and $(10,-1)$.  Note that $c_{n-3}=\FR{n-2}2-t$.
In the even cases, $c_{n-3}=n/2$.  When $n=10$, five $7$-vertices are
together incident to at least $25$ edges, which is more than $\FR12\CH{10}2$.

When $n=8$, four $5$-vertices in $G$ are together incident to at
least $14$ edges, which is the maximum allowed, so there can be no other edges
or other $5$-vertices, the four $5$-vertices induce $K_4$, and
eight edges join these vertices to the rest.  Since
$c_{n-4}=-5$, in $H$ the vertices of degree at least $4$ already contribute
$26$ to the degree-sum, so $H$ has no $3$-vertex.  With $c_{n-5}=0$,
also $G$ has no $3$-vertex.  Hence the degree list of $G$ is
$(5,5,5,5,2,2,2,2)$.  With $(c_5,c_4,c_3)=(4,-5,0)$, applying~\eqref{phidiff}
with $j=2$ now yields $c_2=5$, a contradiction.

When $n=7$, we have $t=-1/2$, and $(c_2,c_3,c_4,c_5,c_6)=(-2,-2,3,-1,0)$.
Hence $\SE i26 ic_i= -3$, and having equal degree-sum requires $c_1=3$.  Now
$H$ has six vertices with degrees $(5,3,3,2,2,0)$ and $G$ has six vertices with
degrees $(4,4,4,1,1,1)$, and they each have one more vertex of the same odd
degree.  Since the degree list of $G$ must be realizable, the only choice is
$(4,4,4,3,1,1,1)$ for $G$ and $(5,3,3,3,2,2,0)$ for $H$.  Now $G$ is realized
only by adding three pendant edges to $K_4$, so $K_4$ is a card in $\cD$, which
can be obtained from $H$ only on the four vertices of high degree.  Thus $H$
consists of copies of $K_4$ and $K_3$ sharing one vertex, plus an isolated
vertex.  Being the union of three complete graphs, $H$ has no independent set
of size $4$, but $G$ does have such a set, so their $4$-decks cannot be equal.

%

\bigskip
{\bf Case 3:} $h=n-1$.  If $c_{n-2}\ge \FR{n+1}3$, then
$m\ge \FR{n+1}3(n-2)-\FR12\FR{n+1}3\FR{n-2}3=\FR 5{18}(n+1)(n-2)$.
Since this exceeds $\FR12\CH n2$ when $n\ge 7$, we conclude $c_{n-2}\le n/3$.

Let $r=-c_{n-1}$.  If $r\ge2$, then~\eqref{j4} and $c_{n-2}\le n/3$ together
yield $c_{n-4}+(n-3)c_{n-3}\ge 2\CH{n-1}3-\FR n3\CH{n-2}2=\FR{(n-2)^2(n-3)}6$.
The contribution to degree-sum in $G$ from vertices of degrees $n-4$ and $n-3$
is now at least $\FR{(n-2)^2(n-3)}6$, which exceeds $\CH n2$ when $n\ge8$.
Hence $n=7$, but then having two $6$-vertices in $H$ requires at
least $11$ edges (more than $\FR12\CH 72$), a contradiction.  Thus we may
assume $r=1$.

With $r=1$, \eqref{j4} yields
$c_{n-4}+(n-3)c_{n-3}+\CH{n-2}{2}c_{n-2}=\CH{n-1}{3}$.  If $c_{n-2}\leq0$, then
$c_{n-4}+(n-3)c_{n-3}\geq \binom{n-1}{3}$.  Dividing by $n-3$ and using
$c_{n-4}+c_{n-3}\le n$ yields $n\ge \FR{(n-1)(n-2)}6$, which requires $n<9$.
If $n=8$, then $c_6\le0$ simplifies \eqref{j4} to $c_4+5c_5\ge35$, but
$a_i\ge c_i$ and $m\le \FR12\CH n2$ yield $28\ge 4a_4+5a_5\ge c_4+5c_5$.  
If $n=7$, then~\eqref{j4} simplifies to $c_3+4c_4\geq20$, but $m\le10$ yields
$3a_3+4a_4\le 20$.  Since $a_i\ge c_i$, we conclude $c_3\le 0$ and $c_4\ge5$.
With at most $10$ edges, $G=K_5+2K_1$.  Now $\cD$ has five cards that are
$K_4$.  With only four edges not incident to its dominating vertex, $H$ cannot
have five such cards.  We conclude $c_{n-2}\ge1$.

With $a_{n-2}\geq c_{n-2}\geq1$, we now break into subcases by the value
of $c_{n-2}$.  We have already proved $c_{n-2}\le n/3$.  Let
$x=\FR{n-1}3-c_{n-2}$, so $x\ge -1/3$ and~\eqref{j4} yields
\begin{equation}\label{x}
c_{n-4}+(n-3)c_{n-3}=x\CH{n-2}2.
\end{equation}
Substituting~\eqref{x} into~\eqref{j5} yields
\begin{equation}\label{n5}
c_{n-5}=\frac{1}{72}(n-3)(n-4)[36c_{n-3}-(n-2)(24x+n-1)]. 
\end{equation}

\smallskip
{\it Subcase 3.1: $x\ge1$.}
If $c_{n-3}\ge\FR{n-2}2$, then with $a_{n-2}\ge1$ the vertices of degrees $n-2$
and $n-3$ in $G$ are incident to at least $\FR {n-2}2(n-3)+(n-2)-\CH{n/2}2$
edges.  Hence $m\ge \FR{(n-2)(3n-4)}8$; this exceeds $\FR12\CH n2$ when $n\ge7$.
If $c_{n-3}\le \FR{n-3}2$, then
$c_{n-4}\geq\CH{n-2}2-(n-3)\FR{n-3}2=\FR{n-3}2$, by~\eqref{x}.  Also
$c_{n-3}\ge1$, since otherwise~\eqref{x} yields $c_{n-4}\ge\CH{n-2}2\ge n$.
If $a_{n-3}=1$, then $c_{n-4}\ge\CH{n-2}2-(n-3)=\FR{(n-3)(n-4)}2$,
again too many vertices when $n\ge7$ (since $a_{n-2}\ge1$).

Hence $a_{n-3}\ge2$.  Now $m\ge {\FR{n+3}2}(n-4)+4-\CH{(n+3)/2}2$.
This quantity exceeds $\FR12\CH n2$ when $n\ge9$.  For $n=8$, we have 
$a_4\ge 3$, $a_5\ge 2$, $a_6\ge1$, yielding degree-sum already $28$,
so $G$ has degree list $(6,5,5,4,4,4,0,0)$, but degree $6$ forbids two
isolated vertices.  For $n=7$, we have $a_{n-4}\ge2$, so even degree-sum
at most $20$ requires degree list $(5,4,4,3,3,1,0)$.  To avoid higher
degree-sum, $a_i=c_i$ for $i\in\{5,4,3,1\}$.  Hence $b_i=0$ for these
values.  Now $H$ having one $6$-vertex requires $b_2=7$ to reach
degree-sum $20$, contradicting $n=7$.

\smallskip
{\it Subcase 3.2: $x\in\{\FR23,\FR13\}$.}
If $c_{n-3}\le 2$, then $c_{n-5}<-n$ when $n\ge9$ by~\eqref{n5},
a contradiction.  If $x=\FR23$, then $c_{n-2}=\FR{n-3}3\in\NN$, so $n\ge9$.
If $x=\FR13$, then $c_{n-2}=\FR{n-2}3\in\NN$, so $n\ge8$.
Setting $n=8$ and $x=\FR13$ and $c_{n-3}\le1$ in~\eqref{n5} yields
$c_{n-5}\le -15$, so $c_{n-3}=2$.  Now~\eqref{x} yields $c_{n-4}+5\cdot 2=5$,
so $c_4=-5$.  With $c_3+3c_4+5c_5+5c_6=0$ by~\eqref{j5}, we have $c_3=-5$.
Now $H$ has at least $10$ vertices, a contradiction.

Hence $c_{n-3}\geq3$.  Since also $c_{n-2}\geq\frac{n-3}{3}$,
the number of edges in $G$ incident to vertices of degree at least
$n-3$ is at least $\FR{n+6}3(n-2)-3-\CH{(n+6)/3}2$, which
simplifies to $\FR5{18}(n+6)(n-3)-3$ and is more than $\FR12\CH n2$ when 
$n\ge8$.

\smallskip
{\it Subcase 3.3: $x=0$.}
Note that $c_{n-2}=\frac{n-1}{3}\in \mathbb{Z}$.  By~\eqref{x},
$c_{n-4}=-(n-3)c_{n-3}$, so $-1\leq c_{n-3}\leq1$.  By~\eqref{n5},
$c_{n-5}= \FR{(n-3)(n-4)}{72}[36c_{n-3}-(n-1)(n-2)]$.  With $c_{n-3}\le1$,
this yields $c_{n-5}<-n$ when $n\ge10$, a contradiction.  Since
$c_{n-2}\equiv 1\mod3$, only $n=7$ remains.

With $n=7$, the expressions above reduce to
$c_5=2$, $c_3=-4c_4$, and $c_2=6c_4-5$, with $-1\leq c_4\leq1$.
If $c_4=-1$, then $c_2=-11<-7$.  If $c_4=1$, then $G$ has three vertices
of degrees $4$ and $5$ such that the number of edges incident to them is at
least $3\cdot 4+2-\CH 32$, which equals $11$ and exceeds $\FR12\CH 72$.

The remaining case is $c_4=c_3=0$ and $c_2=-5$, also $c_5=2$ and $c_6=-1$.
Since we know the $2$-deck, $G$ and $H$ have the same degree-sum;
that is, $\SE i06 ic_i=0$.  We have $\SE i06 ic_i = c_1-10+10-6$; hence
$c_1=6$.  Now $a_5\ge2$ and $a_1\ge6$, which contradicts $n=7$.

\smallskip
{\it Subcase 3.4:} $x=-\FR13$.
Here $c_{n-2}=\frac{n}{3}$, so $n\ge9$.  The number of edges incident to 
vertices of degree at least $n-2$ in $G$ is at least
$\FR n3(n-2)-\FR12\FR n3\FR{n-3}3$, which exceeds $\FR12\CH n2$ when $n>9$ and
equals it when $n=9$.  For $n=9$ with $x=-\FR13$, \eqref{x} reduces to
$c_{5}+6c_6=-7$ and \eqref{n5} reduces to $c_4=15c_6$, which requires $c_6=0$.
Hence $b_5\ge -c_5=7$, which with $b_8=1$ gives $H$ degree-sum at least $43$,
contradicting $m=18$.
\end{proof}

Using Theorem~\ref{degn3}, we present an alternative proof of the result by
Chernyak on the threshold for $2$-reconstructibility of the degree list.

\begin{corollary}[Chernyak~\cite{Che}]\label{deg6}
The degree list of an $n$-vertex graph is $2$-reconstructible whenever
$n\ge6$, and this is sharp.
\end{corollary}

\begin{proof}
Since the $(n-2)$-deck determines the $(n-3)$-deck, it is immediate from
Theorem~\ref{degn3} that the degree list is $2$-reconstructible when $n\ge7$.
By the example of $C_4+K_1$ and $K'_{1,3}$, $n\ge5$ is not sufficient.
It remains only to consider $n=6$.

Let $G$ and $H$ be two $6$-vertex graphs having the same $4$-deck $\cD$ but
different degree lists.  Let $m=\C{E(G)}=\C{E(H)}$ (we know the $2$-deck).
Since the $k$-deck determines the $k$-deck of the complement and $\CH 62=15$,
we may assume $m\le7$.  Define $a_i,b_i,c_i,h$ as in Theorem~\ref{degn3}.
That is, with $k=4$, different degree lists in $G$ and $H$ require a largest
$h$ with $h\ge k$ such that $a_h\ne b_h$, and by symmetry we have
$c_h=a_h-b_h<0$.  We use the equation for $\phi(3)$, which counts
dominating vertices in the cards of the $4$-deck:
\begin{equation}\label{2rec}
c_{3}+4c_4+10c_5=0 .
\end{equation}

\medskip
{\bf Case 1:} $h=5$.
We have $-c_5=1$, because two $5$-vertices in $H$ already force
$m\ge9$.  Thus $4c_4+c_3=10$, by~\eqref{2rec}.  If $c_4\geq3$, then 
$m\ge 3\cdot 4-\CH 32=9$.  If $c_4=2$, then also $c_3=2$ and
$m\ge 2\cdot 4+2\cdot 3-\CH 42=8$.  However, $m\le 7$.  If
$c_4<2$, then $G$ has too many vertices.

\medskip
{\bf Case 2:} $h=4$.  Here $c_3=-4c_4$.  With $n=6$, we have $c_4=-1$ and
$c_3=4$.  With degree-sum at most $14$, the degree list of $G$ is
$(3,3,3,3,x,y)$ with $(x,y)\in\{(2,0),(1,1),(0,0)\}$.  Thus also $b_4=1$ and
$b_3=0$, so $H$ has only one vertex with degree exceeding $2$.  If
$(x,y)=(0,0)$, then $G=K_4+2K_1$ and $K_4$ is a card, but $K_4$ is not
contained in $H$.

Hence $m=7$, and the degree list of $H$ must be $(4,2,2,2,2,2)$.  The only such
graph consists of a $4$-cycle and a $3$-cycle with one common vertex.  Every
card of $H$ has at most four edges.  Whether $(x,y)$ is $(1,1)$ or $(2,0)$,
deleting from $G$ the two vertices of smallest degree eliminates at most two
edges and leaves a card with five edges, a contradiction.
\end{proof}

\section{$3$-reconstructibility of connectedness}

Using Theorem \ref{degn3}, we prove Theorem \ref{conn}.  Again the example of
$C_5+P_1$ and $K''_{1,3}$ shows that the threshold on $n\ge7$ is sharp; they
have the same $3$-deck, but only one is connected.

\begin{theorem*}[{\bf 1.5}]
For $n\ge7$, connectedness is $3$-reconstructible for $n$-vertex graphs,
and the threshold on $n$ is sharp.
\end{theorem*}
\begin{proof}
Suppose that $n$-vertex graphs $G$ and $H$ have the same $(n-3)$-deck $\cD$,
but that $G$ is connected and $H$ is disconnected.  Let $m$ be the common
number of edges in $G$ and $H$.  Let $C$ be the largest component in $H$.
Since $G$ is connected, it has a spanning tree $T$.  Since $n\ge7$,
$T$ has at least two connected cards.  Thus $\cD$ has at least two connected
cards, so $C$ has at least $n-2$ vertices.

By Theorem~\ref{degn3}, $G$ and $H$ have the same degree list.  Since $G$ is
connected, $H$ cannot have an isolated vertex, so $H=C+K_2$.  If $C$ has a
$1$-vertex, then deleting it and the vertices of the small component in $H$
leaves a card in $D$ with $m-2$ edges.  However, since $G$ is connected, it is
not possible to delete three vertices in $G$ and only remove two edges.  Hence
$C$ has no $1$-vertex, which means that $G$ and $H$ each have exactly two
$1$-vertices.  Let $u$ and $v$ be the $1$-vertices in $G$, and let $Y$ be the
set of $1$-vertices in $H$.

Let $x$ be the number of $2$-vertices in both $G$ and in $H$.
If $x=0$, then $C$ has minimum degree at least $3$.  Deleting $Y$ and one
vertex of $C$ from $H$ now yields $n-2$ cards with minimum degree at least $2$.
Such cards can arise from $G$ only by deleting the two $1$-vertices
and one other vertex.  Hence $G-\{u,v\}$ and $C$ have the same $(n-3)$-deck.
They must therefore have the same number of edges.  However, $C$ has $m-1$
edges, while $G-\{u,v\}$ has $m-2$ edges.  Thus $x>0$.

To eliminate only three edges from $H$ when deleting three vertices, one
must delete $Y$ and a $2$-vertex of $C$.  Thus $x$ is also the number
of cards in $\cD$ with $m-3$ edges.  We show the remaining possibilities for
$G$ in Figure~1.

\begin{figure}[hbt]
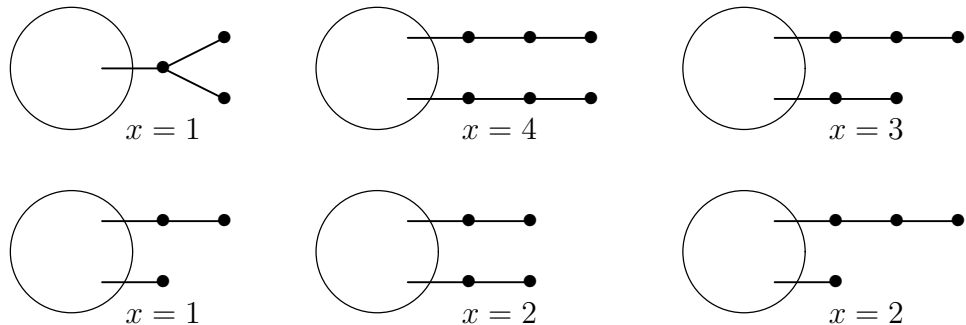

\begin{center}
\gpic{
\expandafter\ifx\csname graph\endcsname\relax \csname newbox\endcsname\graph\fi
\expandafter\ifx\csname graphtemp\endcsname\relax \csname newdimen\endcsname\graphtemp\fi
\setbox\graph=\vtop{\vskip 0pt\hbox{%
    \graphtemp=.5ex\advance\graphtemp by 0.320in
    \rlap{\kern 0.800in\lower\graphtemp\hbox to 0pt{\hss $\bu$\hss}}%
    \graphtemp=.5ex\advance\graphtemp by 0.480in
    \rlap{\kern 1.120in\lower\graphtemp\hbox to 0pt{\hss $\bu$\hss}}%
    \graphtemp=.5ex\advance\graphtemp by 0.160in
    \rlap{\kern 1.120in\lower\graphtemp\hbox to 0pt{\hss $\bu$\hss}}%
    \special{pn 11}%
    \special{pa 800 320}%
    \special{pa 1120 480}%
    \special{fp}%
    \special{pa 800 320}%
    \special{pa 1120 160}%
    \special{fp}%
    \special{pa 800 320}%
    \special{pa 480 320}%
    \special{fp}%
    \special{pn 8}%
    \special{ar 320 320 320 320 0 6.28319}%
    \graphtemp=.5ex\advance\graphtemp by 0.656in
    \rlap{\kern 0.800in\lower\graphtemp\hbox to 0pt{\hss $x=1$\hss}}%
    \graphtemp=.5ex\advance\graphtemp by 1.440in
    \rlap{\kern 0.800in\lower\graphtemp\hbox to 0pt{\hss $\bu$\hss}}%
    \graphtemp=.5ex\advance\graphtemp by 1.120in
    \rlap{\kern 0.800in\lower\graphtemp\hbox to 0pt{\hss $\bu$\hss}}%
    \graphtemp=.5ex\advance\graphtemp by 1.120in
    \rlap{\kern 1.120in\lower\graphtemp\hbox to 0pt{\hss $\bu$\hss}}%
    \special{pn 11}%
    \special{pa 480 1440}%
    \special{pa 800 1440}%
    \special{fp}%
    \special{pa 480 1120}%
    \special{pa 1120 1120}%
    \special{fp}%
    \special{pn 8}%
    \special{ar 320 1280 320 320 0 6.28319}%
    \graphtemp=.5ex\advance\graphtemp by 1.616in
    \rlap{\kern 0.800in\lower\graphtemp\hbox to 0pt{\hss $x=1$\hss}}%
    \graphtemp=.5ex\advance\graphtemp by 0.480in
    \rlap{\kern 2.400in\lower\graphtemp\hbox to 0pt{\hss $\bu$\hss}}%
    \graphtemp=.5ex\advance\graphtemp by 0.480in
    \rlap{\kern 2.720in\lower\graphtemp\hbox to 0pt{\hss $\bu$\hss}}%
    \graphtemp=.5ex\advance\graphtemp by 0.480in
    \rlap{\kern 3.040in\lower\graphtemp\hbox to 0pt{\hss $\bu$\hss}}%
    \graphtemp=.5ex\advance\graphtemp by 0.160in
    \rlap{\kern 2.400in\lower\graphtemp\hbox to 0pt{\hss $\bu$\hss}}%
    \graphtemp=.5ex\advance\graphtemp by 0.160in
    \rlap{\kern 2.720in\lower\graphtemp\hbox to 0pt{\hss $\bu$\hss}}%
    \graphtemp=.5ex\advance\graphtemp by 0.160in
    \rlap{\kern 3.040in\lower\graphtemp\hbox to 0pt{\hss $\bu$\hss}}%
    \special{pn 11}%
    \special{pa 2080 480}%
    \special{pa 3040 480}%
    \special{fp}%
    \special{pa 2080 160}%
    \special{pa 3040 160}%
    \special{fp}%
    \special{pn 8}%
    \special{ar 1920 320 320 320 0 6.28319}%
    \graphtemp=.5ex\advance\graphtemp by 0.656in
    \rlap{\kern 2.560in\lower\graphtemp\hbox to 0pt{\hss $x=4$\hss}}%
    \graphtemp=.5ex\advance\graphtemp by 1.440in
    \rlap{\kern 2.400in\lower\graphtemp\hbox to 0pt{\hss $\bu$\hss}}%
    \graphtemp=.5ex\advance\graphtemp by 1.120in
    \rlap{\kern 2.400in\lower\graphtemp\hbox to 0pt{\hss $\bu$\hss}}%
    \graphtemp=.5ex\advance\graphtemp by 1.440in
    \rlap{\kern 2.720in\lower\graphtemp\hbox to 0pt{\hss $\bu$\hss}}%
    \graphtemp=.5ex\advance\graphtemp by 1.120in
    \rlap{\kern 2.720in\lower\graphtemp\hbox to 0pt{\hss $\bu$\hss}}%
    \special{pn 11}%
    \special{pa 2080 1440}%
    \special{pa 2720 1440}%
    \special{fp}%
    \special{pa 2080 1120}%
    \special{pa 2720 1120}%
    \special{fp}%
    \special{pn 8}%
    \special{ar 1920 1280 320 320 0 6.28319}%
    \graphtemp=.5ex\advance\graphtemp by 1.616in
    \rlap{\kern 2.560in\lower\graphtemp\hbox to 0pt{\hss $x=2$\hss}}%
    \graphtemp=.5ex\advance\graphtemp by 0.480in
    \rlap{\kern 4.320in\lower\graphtemp\hbox to 0pt{\hss $\bu$\hss}}%
    \graphtemp=.5ex\advance\graphtemp by 0.480in
    \rlap{\kern 4.640in\lower\graphtemp\hbox to 0pt{\hss $\bu$\hss}}%
    \graphtemp=.5ex\advance\graphtemp by 0.160in
    \rlap{\kern 4.320in\lower\graphtemp\hbox to 0pt{\hss $\bu$\hss}}%
    \graphtemp=.5ex\advance\graphtemp by 0.160in
    \rlap{\kern 4.640in\lower\graphtemp\hbox to 0pt{\hss $\bu$\hss}}%
    \graphtemp=.5ex\advance\graphtemp by 0.160in
    \rlap{\kern 4.960in\lower\graphtemp\hbox to 0pt{\hss $\bu$\hss}}%
    \special{pn 11}%
    \special{pa 4000 480}%
    \special{pa 4640 480}%
    \special{fp}%
    \special{pa 4000 160}%
    \special{pa 4960 160}%
    \special{fp}%
    \special{pn 8}%
    \special{ar 3840 320 320 320 0 6.28319}%
    \graphtemp=.5ex\advance\graphtemp by 0.656in
    \rlap{\kern 4.480in\lower\graphtemp\hbox to 0pt{\hss $x=3$\hss}}%
    \graphtemp=.5ex\advance\graphtemp by 1.440in
    \rlap{\kern 4.320in\lower\graphtemp\hbox to 0pt{\hss $\bu$\hss}}%
    \graphtemp=.5ex\advance\graphtemp by 1.120in
    \rlap{\kern 4.320in\lower\graphtemp\hbox to 0pt{\hss $\bu$\hss}}%
    \graphtemp=.5ex\advance\graphtemp by 1.120in
    \rlap{\kern 4.960in\lower\graphtemp\hbox to 0pt{\hss $\bu$\hss}}%
    \graphtemp=.5ex\advance\graphtemp by 1.120in
    \rlap{\kern 4.640in\lower\graphtemp\hbox to 0pt{\hss $\bu$\hss}}%
    \special{pn 11}%
    \special{pa 4000 1440}%
    \special{pa 4320 1440}%
    \special{fp}%
    \special{pa 4000 1120}%
    \special{pa 4960 1120}%
    \special{fp}%
    \special{pn 8}%
    \special{ar 3840 1280 320 320 0 6.28319}%
    \graphtemp=.5ex\advance\graphtemp by 1.616in
    \rlap{\kern 4.480in\lower\graphtemp\hbox to 0pt{\hss $x=2$\hss}}%
    \hbox{\vrule depth1.616in width0pt height 0pt}%
    \kern 5.000in
  }%
}%
}
\end{center}

\vspace{-1pc}
\caption{Possibilities for $G$ in Theorem~\ref{conn}.}
\end{figure}

If $u$ and $v$ have the same neighbor, $w$, then $G$ can have a card with $m-3$
edges only if $w$ has degree $3$ and the deleted set is $\{u,v,w\}$.  Hence in
this case $x=1$.

If $u$ and $v$ have different neighbors, then each of $u$ and $v$ is the end of
a maximal path containing no vertices of degree larger than $2$ in $G$; call
these paths $P(u)$ and $P(v)$.  We can only obtain a card with $m-3$ edges by
deleting $i$ vertices from $P(u)$ and $j$ vertices from $P(v)$, where $i+j=3$.
There are at most four choices for $i$, so $x\le4$.  In order to have exactly
$x$ cards with $m-3$ edges, there must be a total of $x$ vertices of degree $2$
on $P(u)\cup P(v)$ and hence no $2$-vertices elsewhere in $G$ (See Figure~1).

%
%
%
%
%
%
%
%
%
%

%


Now consider the cards of $G$ obtained by removing three vertices.  When
$x\ge2$, the paths $P(u)$ and $P(v)$ together have at least four vertices of
degree at most $2$, so removing any three vertices of $G$ leaves a vertex of
degree at most $1$.  Hence removing $Y$ and a vertex of $C$ from $H$ must also
leave a vertex of degree at most $1$.  This means that every vertex of $C$ has
a neighbor of degree $2$.  In the two possibilities when $x=1$, the one card of
$G$ with $m-3$ edges may have no vertex of degree at most $1$, but all other
cards must have such a vertex.  In this case every vertex of $C$ except
possibly one has a neighbor of degree $2$.

For $x\in\{3,4\}$, label $u$ and $v$ so that $|V(P(u))|\geq |V(P(v))|$.
Consider a card $D$ of $G$ with $m-3$ edges that is obtained by deleting $u$,
$v$ and the neighbor of $u$, so $D$ has two vertices of degree $1$ and $x-3$
vertices of degree $2$.  Since all $2$-vertices in $G$ are in $P(u)\cup P(v)$,
the other vertices in $D$ have degree at least $3$.  Note that $D$ must be a
vertex-deleted subgraph of $C$, since cards with $m-3$ edges are obtained from
$H$ only by deleting $Y$ and a vertex of $C$.  Since $C$ must have $x$ vertices
of degree $2$ and none of degree $1$, it must be formed from $D$ by adding one
vertex $z$ of degree $2$ whose neighbors are the two $1$-vertices in $D$.
Adding $z$ to form $C$ shows that the $2$-vertices in $C$
lie along a single path.  This means that only two vertices outside this path
can have neighbors of degree $2$.  Since every vertex of $C$ must have a
neighbor of degree $2$, we conclude that $C$ has at most two vertices outside
the path, but then those vertices cannot have degree greater than $2$, a contradiction.

When $x=2$, recall that every vertex in $C$ has a neighbor of degree $2$
(including the vertices of degree $2$).  Each vertex of degree $2$ is a
neighbor of only two vertices.  Hence $2=x\ge (n-2)/2$, so $n\le6$.
Similarly, when $x=1$, all but one vertex of $C$ has a neighbor of degree $2$,
so $1=x\ge (n-3)/2$, yielding $n\le 5$.

We have obtained contradictions in all cases, so such $G$ and $H$ do not exist.
\end{proof}

Using Theorem~\ref{conn}, Manvel's result on $2$-reconstructibilty of
connectedness follows quite easily.

\begin{corollary}[Manvel~\cite{Manvel}]
For $n\ge6$, connectedness of an $n$-vertex graph is $2$-reconstructible.
\end{corollary}

\begin{proof}
Again $C_4+K_1$ and $K'_{1,3}$ give sharpness, and Theorem~\ref{conn} handles
$n\ge7$.  Consider connected and disconnected $6$-vertex graphs $G$ and $H$
with the same $4$-deck.

By Corollary~\ref{deg6}, $G$ and $H$ have the same degree list, so neither
has isolated vertices.  Since $G$ has a connected $4$-card, $H$ has a $4$-vertex
component $C$, and $H=C+K_2$.  Thus $H$ has only one connected $4$-card.

\nobreak
Now $G$ must also have only one connected $4$-card.  Therefore every spanning
tree of $G$ is a path, so $G$ is a path, but then $G$ has three connected
$4$-cards.
\end{proof}

\end{document}